\documentclass[12pt]{article}
\usepackage{graphicx}
\usepackage{color}
\usepackage{amsmath,amsopn,amssymb,amsthm,amsfonts}
\usepackage[T2A]{fontenc}

\newtheorem{theorem}{Theorem}[section]
\newtheorem{corollary}[theorem]{Corollary}

\newtheorem{lemma}[theorem]{Lemma}
\newtheorem{remark}[theorem]{Remark}
\newtheorem{conclusion}[theorem]{Conclusion}

\newcommand\g{{\mathfrak g}}

\newcommand{\R}{\mathbb{R}}

\begin{document}

{\bf \large
\centerline{N.~K.~Smolentsev}

\vspace{3mm}
\centerline{Almost complex and almost para-complex Cayley structures }
\centerline{on six-dimensional pseudo-Riemannian spheres}}
\vspace{3mm}

\begin{abstract}
In this paper we study almost complex and almost para-complex Cayley structures on six-dimensional pseudo-Riemannian spheres in the space of purely imaginary octaves of the split Cayley algebra $\mathbf{Ca}'$. It is shown that the Cayley structures are non-integrable, their basic geometric characteristics are calculated. In contrast to the usual Riemann sphere $\mathbb{S}^6$, there exist (integrable) complex structures and para-complex structures on the pseudospheres under consideration.
\end{abstract}

\section{Introduction} \label{Introduction}
Almost complex structures on the usual six-dimensional sphere $\mathbb{S}^6$ have been studied actively for a long time  (see, for example,
the bibliography and a historical review  in \cite{Agricola-18-1},  \cite{Agricola-18-2}).
However, the  question about existence of the integrable almost complex structures on $\mathbb{S}^6$ has not been solved so far.
Among orthogonal almost complex structures $J$ on $\mathbb{S}^6$, an almost-complex Cayley structure $J_0$ takes on a special place.
It is obtained by means of the vector product in the ambient space $\mathbb{R}^7$ of the purely imaginary Cayley octaves.
The Cayley structure $J_0$ is invariant with respect to the action of the compact exceptional group $G_2$.
A detailed investigation of the Cayley structure $J_0$ on $\mathbb{S}^6$ was done in \cite{Smolen-9}.

As is known, there exists a split Cayley algebra $\mathbf{Ca'}$, which is obtained from quaternions by the Cayley-Dickson construction using an ''imaginary'' number $e$ such that $e^2 = +1$.
The Cayley algebra $\mathbf{Ca'}$ has a pseudo-Euclidean scalar product of signature (4,4) defined by the quadratic form
$x\overline{x}=x_0^2+x_1^2+x_2^2+x_3^2-x_4^2-x_5^2-x_6^2-x_7^2$. The automorphism group of the algebra $\mathbf{Ca'}$ is a noncompact exceptional group $G_2^*$.
The space of purely imaginary split octonions inherits the quadratic form $g$ of the signature (3,4) and it is therefore the pseudo-Euclidean space $\mathbb{R}^{3,4}$.
There are two types of spheres in the space $\mathbb{R}^{3,4}$: of real and imaginary radius.
A pseudo-sphere $\mathbb{S}^{2,4}$ of real radius is a homogeneous pseudo-Riemannian manifold $\mathbb{S}^{2,4}=G_2^*/SU(1,2)$ of signature (2,4).
The scalar squares of the normal vectors to $\mathbb{S}^{2,4}$ are positive.
A pseudo-sphere $\mathbb{S}^{3,3}(i)$ of imaginary radius is a homogeneous pseudo-Riemannian manifold $\mathbb{S}^{3,3}=G_2^*/SL(3,\mathbb{R})$  of signature (3,3).
Scalar squares of the normal vectors to $\mathbb{S}^{3,3}(i)$ are negative.

The multiplication of the purely imaginary octaves defines $G_2^*$-invariant vector cross product in the space $\mathbb{R}^{3,4}$.
This allows us to define the orthogonal almost complex Cayley structure $J$ on the sphere $\mathbb{S}^{2,4}\subset \mathbb{R}^{3,4}$
by multiplying the tangent vectors on the normal vector.

In this paper it is shown that the Cayley structure is nonintegrable and its Nijenhuis tensor is calculated through.
An expression of the fundamental 2-form $\omega$ of the almost Hermitian structure on $\mathbb{S}^{2,4}$ is found,
and it is shown that the 2-form $\omega$ is proper for the Laplace operator. In contrast to the usual Riemannian sphere $\mathbb{S}^{6}$,
there are integrable almost complex structures on $\mathbb{S}^{2,4}$.
On the sphere $\mathbb{S}^{3,3}\subset \mathbb{R}^{3,4}$ of imaginary radius, the vector cross product in $\mathbb{R}^{3,4}$
defines an almost para-complex structure $P$.
It is shown that the Cayley structure $P$ is non-integrable and the Nijenhuis tensor is calculated.
The expression of the fundamental 2-form $\omega$ is found on $\mathbb{S}^{3,3}(i)$ and it is shown that $\omega$ is a proper 2-form of the Laplace operator.
In contrast to the usual Riemannian sphere $\mathbb{S}^{6}$, there are integrable almost para-complex structures on $\mathbb{S}^{3,3}(i)$,
which are easily constructed using stereographic projection.

\section{Preliminaries} \label{Preface}
\subsection{Cayley algebras} \label{Cayley algebras}
Let $\mathbb{H}$ be the quaternion algebra consisting of the numbers $w = x_0 +x_1e_1  +x_2e_2  +x_3e_3$. The algebra $\mathbf{Ca}$
of Cayley numbers is obtained by the Cayley-Dickson construction.
We introduce one more imaginary unit $e$ and form Cayley numbers in the form $x = a+be$, where $a$ and $b$ are quaternions.
Multiplication of such numbers is determined by the formula
$xy =(a+be)(c+de)=(ac-\overline{d}b)+(da+b\overline{c})e.$
The group of automorphisms of the Cayley algebra $\mathbf{Ca}$ is a simple exceptional compact Lie group $G_2$.

Another version of the Cayley-Dickson procedure is known. In this case, an additional ''unit'' $e$ has the property $e^2 = +1$.
Then we obtain, so-called, split-octonions in the form $x = a+be$, where $a$ and $b$ are quaternions. Multiplication of such numbers is determined by the
formula
$$
xy =(a+be)(c+de)=(ac+\overline{d}b)+(da+b\overline{c})e.
$$
As a result, we obtain a non-associative algebra $\mathbf{Ca'}$, which is called the split Cayley algebra.
We introduce the following notation: $e_4 = e$,  $e_5 = e_1e$,  $e_6 = e_2e$, $e_7 = e_3e$. Note that $e_i^2 = -1$ for $i = 1, 2, 3$, and $e_j^2 = +1$ for $j = 4, 5, 6, 7$. Each Cayley split-number is written as $x = x_0 + x_1 e_1 + \dots + x_7e_7$, where $x_\alpha \in \R$, and $e_1, e_2, \dots , e_7$ are imaginary units.
The following rules for their multiplication hold (the first factor in the column on the left, and the second factor in the row above):
%\centerline{\begin{tabular}{|r|r|r|r|r|r|r|r|}
%  \hline
%  % after \\: \hline or \cline{col1-col2} \cline{col3-col4} ...
%       & $e_1$  & $e_2$  & $e_3$  & $e_4$ & $e_5$  & $e_6$  & $e_7$ \\ \hline
% $e_1$ & $-1$   & $e_3$  & $-e_2$ & $e_5$ & $-e_4$ & $-e_7$ & $e_6$ \\ \hline
% $e_2$ & $-e_3$ & $-1$   & $e_1$  & $e_6$ & $e_7$  & $-e_4$ & $-e_5$ \\ \hline
% $e_3$ & $e_2$  & $-e_1$ & $-1$   & $e_7$ & $-e_6$ & $e_5$  & $-e_4$ \\ \hline
% $e_4$ & $-e_5$ & $-e_6$ & $-e_7$ & $1$   & $-e_1$ & $-e_2$ & $-e_3$ \\ \hline
% $e_5$ & $e_4$  & $-e_7$ & $e_6$  & $e_1$ & $1$    & $e_3$  & $-e_2$ \\ \hline
% $e_6$ & $e_7$  & $e_4$  & $-e_5$ & $e_2$ & $-e_3$ & $1$    & $e_1$ \\ \hline
% $e_7$ & $-e_6$ & $e_5$  & $e_4$  & $e_3$ & $e_2$  & $-e_1$ & $1$ \\
%  \hline
%\end{tabular}}

$$
\begin{array}{r|rrrrrrrr}
\cdot& e_1  & e_2  & e_3  & e_4 & e_5  & e_6  & e_7 \\
  \hline
 e_1 & -1   & e_3  & -e_2 & e_5 & -e_4 & -e_7 & e_6 \\
 e_2 & -e_3 & -1   & e_1  & e_6 & e_7  & -e_4 & -e_5 \\
 e_3 & e_2  & -e_1 & -1   & e_7 & -e_6 & e_5  & -e_4 \\
 e_4 & -e_5 & -e_6 & -e_7 & 1   & -e_1 & -e_2 & -e_3 \\
 e_5 & e_4  & -e_7 & e_6  & e_1 & 1    & e_3  & -e_2 \\
 e_6 & e_7  & e_4  & -e_5 & e_2 & -e_3 & 1    & e_1 \\
 e_7 & -e_6 & e_5  & e_4  & e_3 & e_2  & -e_1 & 1 \\
\end{array}
$$
%$$
%\begin{array}{r|rrrrrrrr}
%  \cdot & e_1 & e_2 & e_3 & e_4 & e_5 & e_6 & e_7 \\
%  \hline
%  e_1 &  -1  & e_3 & -e_2 &  e_5 & -e_4 & -e_7 &  e_6 \\
%  e_2 & -e_3 & -1  &  e_1 &  e_6 &  e_7 & -e_4 & -e_5 \\
%  e_3 &  e_2 &-e_1 &  -1  &  e_7 & -e_6 &  e_5 & -e_4 \\
%  e_4 & -e_5 &-e_6 & -e_7 &  -1  &  e_1 &  e_2 &  e_3 \\
%  e_5 &  e_4 &-e_7 &  e_6 & -e_1 &  -1  & -e_3 &  e_2 \\
%  e_6 &  e_7 & e_4 & -e_5 & -e_2 &  e_3 &  -1  & -e_1 \\
%  e_7 & -e_6 & e_5 &  e_4 & -e_3& -e_2  &  e_1 &  -1  \\
%\end{array}
%$$

We recall the basic properties of the algebra $\mathbf{Ca'}$, for more details about this see \cite{Zhevlakov}.
The conjugation of split octonions is given in the usual way, $\overline{x} = x_0 -x_1e_1 -\dots -x_7e_7$, and has the property $\overline{xy}=\overline{y}\,\overline{x}$. The Cayley algebra $\mathbf{Ca'}$ has the quadratic form $N(x)=x\overline{x}=x_0^2+x_1^2+x_2^2+x_3^2-x_4^2-x_5^2-x_6^2-x_7^2$
and the appropriate pseudo-Euclidean scalar product of signature (4,4).
Therefore, the algebra $\mathbf{Ca'}$ will often be regarded as a pseudo-Euclidean space $\R^{4,4}$.
The algebra $\mathbf{Ca'}$ is composition algebra \cite{Zhevlakov}, since for it the equality $N(xy,xy) = N(x,x)N(y,y)$.

The algebra $\mathbf{Ca'}$ is nonassociative, i.e. $(xy)z\neq x(yz)$. The \emph{associator} is the expression $[x,y,z] = (xy)z -x(yz)$.
The split algebra $\mathbf{Ca'}$ is an alternative, since the following properties hold:
$$
(xx)y = x(xy),\quad x(yy) = (xy)y, \quad \forall x,y \in \mathbf{Ca'}.
$$
This name is given to such algebras because the associator is skew-symmetric (alternative) in all arguments. In particular, the property of alternativeity is written through the associator in the following way: $[x,x,y] = 0$, $[x, y, y] = 0$.
We note some other properties of the algebra $\mathbf{Ca'}$:
\begin{itemize}
  \item $x(yx) = (xy)x$,
  \item $(x\overline{x})y=x(\overline{x}y)$, $x(\overline{y}y) = (x\overline{y})y$,
  \item $(xxy)x = xx(yx)$, $[xx, y, x] = 0$,
  \item $((xy)z)y = x(yzy)$, $(xyx)z = x(y(xz))$, $(xy)(zx) =x(yz)x$.
\end{itemize}

\subsection{Group $G_2^*$} \label{Group G2}
The group $G_2$ is one of the complex exceptional simple Lie groups \cite{Yokota-09}.
It has two real forms. One of them, compact, is denoted by the symbols $G_2^c$ or $G_2$ and is well known. The second, noncompact real form is denoted as $G_2^*$, its description can be found in \cite{Yokota77}. This group $G_2^*$ can be defined as the automorphism group of the split algebra $\mathbf{Ca'}$, or as the stabilizer of some general 3-form on the pseudo-Euclidean vector space $(V,g)$ of the signature (3,4), or as the stabilizer of a vector cross product on the seven-dimensional pseudo-Euclidean space $(V,g)$ of the signature (3,4). Therefore, $G_2^*\subset SO(4,3)$.

In the pseudo-Euclidean space $V$ we choose a basis $(e_1, e_2,\dots, e_7)$ in which the quadratic form takes the form: $g=-2(e^1\cdot e^5+e^2\cdot e^6+e^3\cdot e^7)-(e^4)^2$. Then the group $G_2^*$ is \cite{Fino-Kath} the stabilizer of the next 3-form on $V$:
$$
\Omega_0=\sqrt{2}(e^{123}-e^{567}) +e^4\wedge(e^{15}+e^{26}+e^{37}),
$$
where $e^{ijk}=e^{i}\wedge e^{j}\wedge e^{k}$.
The general element $A$ of the Lie algebra $\g_2^*$ of the group $G_2^*$ in the basis $e_i$ has a simple block form:
$$
A= \left[ \begin {array}{ccccccc} -{a_1}-{a_4}&{a_5}&{a_6}&-
\sqrt {2}{a_{12}}&0&-{a_{11}}&{a_{10}}\\
\noalign{\medskip}{a_7}&{a_1}&{a_2}&-\sqrt {2}{a_{13}}&{a_{11}}&0&-{a_9} \\
\noalign{\medskip}{a_8}&{a_3}&{a_4}&-\sqrt {2}{a_{14}}&-{
a_{10}}&{a_9}&0\\
\noalign{\medskip}-\sqrt {2}{a_9}&-\sqrt {2}{a_{10}}&-\sqrt {2}{a_{11}}&0&\sqrt {2}{a_{12}}&\sqrt {2}{a_{13}}&\sqrt {2}{a_{14}}\\
\noalign{\medskip}0&{a_{14}}&-{a_{13}}&\sqrt {2}{a_9}&{a_1}+{a_4}&-{a_7}&-{a_8}\\
\noalign{\medskip}-{a_{14}}&0&{a_{12}}&\sqrt {2}{a_{10}}&-{a_5}&-{a_1}&-{a_3}
\\
\noalign{\medskip}{a_{13}}&-{a_{12}}&0&\sqrt {2}{a_{11}}&-{a_6}&-{a_2}&-{a_4}\end {array} \right]
$$
In \cite{Yokota77} it is shown that the group $G_2^*$ is homeomorphic to the product space $SO(4)\times \R^8$.

\subsection{Vector cross  products} \label{Vector-prod}
Let $V$ be a real $n$-dimensional vector space with nondegenerate bilinear symmetric form $\langle x,y \rangle$.
A multilinear map $P:V^r \to V$, $(1 \leq r \leq n)$ is called the \emph{vector ($r$-fold) cross product} \cite{Gray} on $V$ if it has properties:
$$
\langle P(x_1,\dots, x_r),x_i \rangle =0,\ 1 \leq i \leq r \quad \text{and} \quad ||P(x_1,\dots, x_r)||^2 =\det(\langle x_i,x_j \rangle).
$$

One-fold vector cross products are orthogonal complex structures $J$ on  even-dimensional vector space $V$.
The two-fold vector cross product exists only in dimension 3 and 7.
On a seven-dimensional vector space $V$ they are described as follows \cite{Gray}.
Let $W$ be a composition algebra and $V$ an orthogonal complement to the identity $e$ in the algebra $W$.
Define $P: V\times V \to V$ by the formula $P(x,y) = xy +\langle x,y \rangle$. Then $P$ is a two-fold vector cross product and, conversely, each such vector cross product arises in this way.
If $\dim W = 8$, we obtain a 2-fold vector product on a 7-dimensional space $V$.
The bilinear form associated with such a vector product has the signature (0, 7) or (4, 3). The automorphism group is either $G_2$ or $G_2^*$, respectively.
	
3-fold vector cross works. Let $V$ be a composition algebra and let $\langle x,y \rangle$ be the corresponding bilinear symmetric form.
We define $P_1,P_2: V\times V \to V$ by the formulas \cite{Gray}:
$$
P_1(x,y,z)= \varepsilon(-x(\overline{y}z) +\langle x,y\rangle z +\langle
y,z\rangle x -\langle z,x\rangle y),
$$
$$
P_2(x,y,z)= \varepsilon(-(x\overline{y})z +\langle x,y\rangle z +\langle
y,z\rangle x -\langle z,x\rangle y),
$$
where $\varepsilon=\pm 1$.
Then $P_1$ and $P_2$ are 3-fold vector cross products on $V$ with bilinear form $\varepsilon\langle x,y\rangle$, and vice versa, each 3-fold vector cross product arises in this way.

The Cayley algebra $\mathbf{Ca'}$ is composite with bilinear symmetric form of the signature (4,4), therefore 3-fold vector cross products $P_1$ and $P_2$ are defined on it.
The conjugation operation defines an anti-isomorphism of these vector products. Therefore, in what follows we will consider only the second vector product $P$:
$$
P(x,y,z)= -(x\overline{y})z +\langle x,y\rangle z +\langle
y,z\rangle x -\langle z,x\rangle y.
$$

\subsection{Pseudo-spheres in the space $\R^{3,4}$} \label{Pseudospheres}
Consider the seven-dimensional space $\R^7$ of purely imaginary octaves $X = x_1e_1 +\dots + x_7e_7$ of the split Cayley algebra $\mathbf{Ca'}$.
It inherits the quadratic form $g$ of the signature (3,4) and is therefore the pseudo-Euclidean space $\R^{3,4}$.
In the space $\R^{3,4}$, there are two types of (pseudo) spheres:
$$
x_1^2+x_2^2+x_3^2-x_4^2-x_5^2-x_6^2-x_7^2=r^2, \quad \text{ and } \quad
x_1^2+x_2^2+x_3^2-x_4^2-x_5^2-x_6^2-x_7^2=-r^2.
$$
The first pseudo sphere $\mathbb{S}^{2,4}(r)$ has a real radius $r>0$, it intersects the axes $Ox_1$, $Ox_2$, $Ox_3$.
The tangent planes are pseudo-Euclidean signatures (2.4).
Scalar squares of the normal vectors to $\mathbb{S}^{2,4}(r)$ are positive, and their squares in the Cayley algebra $\mathbf{Ca'}$ are negative. The group $G_2^*$ acts transitively and isometrically on $\mathbb{S}^{2,4}(r)$, and its isotropy group coincides with $SU(1,2)$. Therefore, the sphere $\mathbb{S}^{2,4}(r)$ is a homogeneous space: $\mathbb{S}^{2,4}=G_2^*/SU(1,2)$. We also note that $\mathbb{S}^{2,4}=SO(3,4)/SO(2,4)$.

The second pseudo sphere $\mathbb{S}^{3,3}(ir)$ has imaginary radius $ir$, it intersects the axes $Ox_4$, $Ox_5$, $Ox_6$ and $Ox_7$.
The tangent planes are pseudo-Euclidean signatures (3.3).
Scalar squares of the normal vectors to $\mathbb{S}^{3,3}(ir)$ are negative, and their squares in the Cayley algebra $\mathbf{Ca'}$ are positive.
The group $G_2^*$ acts transitively and isometrically on $\mathbb{S}^{3,3}(ir)$, and its isotropy group coincides with $SL(3,\R)$. Therefore, the sphere $\mathbb{S}^{3,3}(ir)$ is a homogeneous space: $\mathbb{S}^{3,3}=G_2^*/SL(3,\R)$. .

\subsection{Almost para-complex structures} \label{para-complex}
An \emph{almost para-complex} structure on a $2n$-dimensional manifold $M$ is the field $P$ of endomorphisms of the tangent bundle $TM$ such that $P^2 = Id$ and the ranks of the eigen-distributions $T^\pm := \mathrm{ker}(Id \mp P)$, corresponding to the eigenvalues $\pm 1$, are equal \cite{Aleks-09}.
An almost para-complex structure $P$ is said to be \emph{integrable} if the eigen-distributions $T^\pm$ are involutive.
In this case $P$ is called a \emph{para-complex} structure.
The Nijenhuis tensor $N_P$ of an almost para-complex structure $P$ is defined by formula
$$
N_P(X,Y) = 2([X,Y] +[PX,PY] -P[PX,Y] -P[X,PY]),
$$
for all vector fields $X, Y$ on $M$.
As in the complex case, the almost para-complex structure $P$ is integrable if and only if $N_P = 0$.
In \cite{Aleks-09}, a review of the theory is presented and invariant para-complex and para-K\"{a}hler structures on Lie groups are considered in detail.

\section{The vector cross product in $\R^{3,4} = \Im(\mathbf{Ca'})$} \label{vector-cross-product}
Let $\mathbf{Ca'}$ be the split Cayley algebra.
As usual, numbers from $\mathbf{Ca'}$ of the form $x =x_0\in \R$ will be called real numbers, and the numbers $X = x_1e_1 +x_2e_2 +\dots +x_7e_7$ are purely imaginary.
We will write down the split octonions in the form of a sum $x = x_0 +X$.
The space $\R^{3,4} = \Im(\mathbf{Ca'})$ of imaginary octaves inherits from $\mathbf{Ca'}$ the scalar product $g(X,Y) = \langle X,Y \rangle$ of signature (3,4).
We define the vector cross product of the elements $X,Y\in \Im(\mathbf{Ca'})$ as the imaginary part of their product in the algebra $\mathbf{Ca'}$:
$$
X\times Y = \Im(XY).
$$
It is easy to see that
$$
X\times Y = XY +g(X,Y).
$$
It is also easy to see that the vector cross product $X\times Y$ is bilinear, skew-symmetric, and orthogonal to each of their factors.
The multiplication in the algebra $\mathbf{Ca'}$ is expressed in terms of the vector cross product as follows, for $x = x_0 +X$ and $y =y_0 +Y$, we have:
$$
xy = (x_0 y_0 - \langle X,Y \rangle) +x_0Y +y_0X + X\times Y.
$$
For any vectors $X, Y \in \Im(\mathbf{Ca'})$, the following equality holds:
$$
X\times (X\times Y) = -g(X,X)Y + g(X,Y)X.
$$
In particular, if $\mathbf{n}$ is a vector of unit length, then for any $Y\in \R^{3,4}$,
$$
\mathbf{n}\times (\mathbf{n}\times Y) = -Y + g(\mathbf{n},Y)\mathbf{n}.
$$
The \emph{scalar triple product} is defined by the equality $(XYZ) = g(X,Y\times Z) = g(X\times Y,Z)$ and it is a skew-symmetric 3-form $\Omega$ on $\R^{3,4}$,
$$
\Omega(X,Y,Z) = g(X\times Y,Z).
$$
In the notation $\omega^{pqr} = dx^p\wedge dx^q\wedge dx^r$ the 3-form $\Omega$ has the following expression:
$$
\Omega = \omega^{123} -\omega^{145} +\omega^{167} -\omega^{246} -\omega^{257} -\omega^{347} +\omega^{356}.
$$
Since $\R^{3.4}$ has a pseudo-Riemannian metric and the corresponding volume form $\mu(g)=\omega^{1234567}$, the Hodge $*$-operator is defined, $*:\Lambda^k(\R^{3.4})\to  \Lambda^{7-k}(\R^{3.4})$.
Therefore, on the space $\R^{3.4}= \Im (\mathbf{Ca'})$ there is defined the 4-form $\Psi = *\Omega$,
$$
\Psi = \omega^{4567} -\omega^{2367} +\omega^{2345} -\omega^{1357} -\omega^{1346} -\omega^{1256} +\omega^{1247}.
$$
Then $\Omega = *\Psi$.
It is easy to see that the following expression holds on the vectors $X,Y,Z,W \in \R^{3,4}$:
$$
\Psi(X,Y,Z,W) = g(X,(Y\times Z)\times W) = -g(X,Y\times (Z\times W)).
$$

\begin{lemma}\label{lem-1-associator}
On the space $\R^{3,4} = \Im(\mathbf{Ca'})$, the associator $[X,Y,Z]$ can be expressed by the following formula:
$$
 [X,Y,Z] = 2(X\times Y)\times Z + 2g(Y,Z)X -2g(Z,X)Y.
$$
\end{lemma}
\begin{proof}
Using the skew-symmetry of the 3-form $\Omega$ and the associator $[X,Y,Z]$, and also the formula $XY = X\times Y -g(X,Y)$, we obtain:
\begin{multline*}
[X,Y,Z] = (XY)Z - X(YZ) =  \\
=(X\times Y)\times Z - X\times (Y\times Z) -g(X\times Y, Z) +g(X,Y\times Z) -g(X,Y)Z + g(Y,Z)X =\\
=  (X\times Y)\times Z -X\times (Y\times Z) -g(X,Y)Z + g(Y,Z)X.
\end{multline*}
Now we use the resulting expression in the following sum and obtain the necessary formula:
$$
 [X,Y,Z] = [X,Y,Z] + [Z,X,Y] - [X,Z,Y] =
$$
$$
= 2(X\times Y)\times Z + 2g(Y,Z)X -2g(Z,X)Y.
$$
\end{proof}

\begin{corollary}  The following formula holds:
$$
(X\times Y)\times Z -g(X,Z)Y + g(Y,Z)X = -X\times (Y\times Z) + g(X,Z)Y -g(X,Y)Z.
$$
\end{corollary}
\begin{proof}
It follows from $[X,Y,Z] = -[Z,Y,X]$.
\end{proof}

We also need the following properties of the vector cross product, which immediately follow from the preceding equality:
\begin{itemize}
  \item If $\mathbf{n},Y,Z \in \R^{3,4}$ and if $Y,Z \perp \mathbf{n}$, then $$ (\mathbf{n}\times Y)\times Z = -\mathbf{n}\times (Y\times Z) -g(Y,Z)\mathbf{n}.$$

  \item If $\mathbf{n}$ is a vector of unit length, then for any $Y,Z \in \R^{3,4}$:
$$
\mathbf{n}\times (\mathbf{n}\times Z) = -Z + g(\mathbf{n},Z)\mathbf{n},
$$
$$
g(\mathbf{n}\times Y, \mathbf{n}\times Z) = g(Y,Z) -g(Y,\mathbf{n}) g(Z,\mathbf{n}).
$$
  \item If $g(\mathbf{n},\mathbf{n}) = -1$, then for any $Y,Z \in \R^{3,4}$:
$$
\mathbf{n}\times (\mathbf{n}\times Z) = Z + g(\mathbf{n},Z)\mathbf{n},
$$
$$
g(\mathbf{n}\times Y, \mathbf{n}\times Z) = -g(Y,Z) -g(Y, \mathbf{n}) g(Z, \mathbf{n}).
$$
\end{itemize}

\section{Almost complex structure on $\mathbb{S}^{2,4}$} \label{ACS-on-S24}
Consider the sphere $\mathbb{S}^{2,4} = \mathbb{S}^{2,4}(1)$ of unit radius in the space $\R^{3,4} = \Im(\mathbf{Ca'})$ of imaginary Cayley split-octaves. Let $T_x\mathbb{S}^{2,4}$ be the tangent plane to the sphere at the point $x\in \mathbb{S}^{2,4}$.
It is a pseudo-Euclidean space signature (2.4).
The scalar squares of the normal vectors $\mathbf{n}(x) =x$ to the sphere $\mathbb{S}^{2,4}$ are positive.
Consider the vector cross multiplication of the tangent vectors $Y\in T_x\mathbb{S}^{2,4}$ by the normal vector $\mathbf{n}(x)$.
It is easy to see that this operation takes the tangent space into itself: if $Y\in T_x\mathbb{S}^{2,4}$, that is, $g(Y,\mathbf{n}) = 0$, then
$$
g(\mathbf{n}\times Y,\mathbf{n}) = 0 \Rightarrow \mathbf{n}\times Y \in T_x\mathbb{S}^{2,4}.
$$
It follows from the property of the vector cross product that if $Y\in T_x\mathbb{S}^{2,4}$, then
$$
\mathbf{n}\times (\mathbf{n}\times Y) = -g(\mathbf{n},\mathbf{n})Y +g(\mathbf{n},Y)\mathbf{n} = -Y.
$$
This means that the operator $J_x(Y) = \mathbf{n}\times Y$ of the left multiplication of the tangent vectors at the point $x\in \mathbb{S}^{2,4}$ by the normal vector $\mathbf{n}(x) =x$ to the sphere $\mathbb{S}^{2,4}$ defines a complex structure on $T_x\mathbb{S}^{2,4}$.
Since such an operation is defined at every point $x\in \mathbb{S}^{2,4}$, we obtain that the unit sphere $\mathbb{S}^{2,4}$ has a natural almost complex structure $J$, which we shall call the \emph{Cayley structure}:
$$
J_x: T_x \mathbb{S}^{2,4} \rightarrow T_x \mathbb{S}^{2,4},\quad J_x(Y) = \mathbf{n}(x)\times Y.
$$
From the equality $g(\mathbf{n}\times X, \mathbf{n}\times Y) = g(X,Y) - g(X,\mathbf{n})g(Y,\mathbf{n})$,
it follows immediately that it is orthogonal.
Let $\omega(X,Y) = g(JX,Y)$ be the fundamental 2-form corresponding to $J$. Since $g(X\times Y,Z) = \Omega(X,Y,Z)$, it is easy to see that
$$
\omega(X,Y) = g(\mathbf{n}\times X,Y) = g(\mathbf{n},X\times Y).
$$
\begin{lemma}\label{Lemma 2}
The 3-form $\Omega$ under its restriction to the sphere $\mathbb{S}^{2,4}$ has the property:
$$
\Omega(JX,Y,Z) = \Omega(X,JY,Z) = \Omega(X,Y,JZ).
$$
\end{lemma}
\begin{proof}
We use the formula $(\mathbf{n}\times Y)\times Z = -\mathbf{n}\times (Y\times Z) -g(Y,Z)\mathbf{n}$ and equality $g(X,Y\times Z) = g(X\times Y,Z)$. $\Omega(Z, JX,Y) = g(Z, JX\times Y) = g(Z, (\mathbf{n}\times X)\times Y) = g(Z, -\mathbf{n}\times (X\times Y) -g(X,Y)\mathbf{n}) = g(Z, -\mathbf{n}\times (X\times Y)) =  -g(Z\times \mathbf{n}, X\times Y) = g(\mathbf{n}\times Z, X\times Y) = g(JZ, X\times Y) = \Omega(JZ,X,Y)$.
\end{proof}

\begin{lemma}\label{Lemma 3}
For any $X,Y,Z \in T_x \mathbb{S}^{2,4}$, we have the equality
$$
\iota_\mathbf{n}\Psi(X,Y,Z) = -\Omega(JX,Y,Z),
$$
where $\iota_\mathbf{n}$ is the inner product with the normal vector $\mathbf{n}(x)$.
\end{lemma}
\begin{proof}
$\iota_\mathbf{n}\Psi(X,Y,Z) = -g(\mathbf{n},X\times (Y\times Z)) = -g(\mathbf{n}\times X,Y\times Z) =-\Omega(JX,Y,Z)$.
\end{proof}

Let $e_1,\dots, e_6$ be an orthonormal basis of the space $T_x \mathbb{S}^{2,4}$, in which $e_1$ and $e_2$ are space-like and the others are time-like. Then $\mathbf{n}(x), e_1,\dots, e_6$ is the orthonormal basis of the space $\R^{3,4}= \R\{\mathbf{n}(x)\} \oplus T_x\mathbb{S}^{2,4}$.
Let $\mu =\mathbf{n}^*\wedge e^1 \wedge \dots \wedge e^6$ be the volume element of $\R^{3,4}$ and $\mu_S = e^1 \wedge \dots \wedge e^6$ be the volume element of the sphere $\mathbb{S}^{2,4}$ (at the point $x$).
Obviously, $\mu_S =\iota_\mathbf{n}\mu$, where $\iota_\mathbf{n}$ is the inner product with the normal vector $\mathbf{n}(x)$.
Let $\ast_S$ and $\ast_R$ be Hodge operators on $\mathbb{S}^{2,4}$ and  $\R^{3,4}$, respectively.
The symbol $\theta|_S$ denotes the restriction of the differential form $\theta$ in $\R^{3,4}$ to the submanifold $\mathbb{S}^{2,4}$ and the symbol $\theta_\mathbf{n}$ is the normal component of the form $\theta$, that is, $\theta_\mathbf{n}=\mathbf{n}^*\wedge \widetilde{\theta}$, where $\widetilde{\theta}$ is a $(k-1)$-form, which is expressed in terms of $e^1,\dots, e^6\in T_x^*\mathbb{S}^{2,4}$. Then $\theta=\theta|_S + \mathbf{n}^*\wedge \widetilde{\theta}$.

%\begin{lemma}\label{Lemma 4}
%Let $\theta$ be a $k$-form on $\mathbb{S}^{2,4}$.
%Then
%$$
%\ast_R\theta = (-1)^k\mathbf{n}\wedge \ast_S\theta, \quad  \ast_S\theta = (-1)^k\iota_\mathbf{n}(\ast_R\theta).
%$$
%\end{lemma}
%\begin{proof}
%For any $k$-form $\varphi$ on $\mathbb{S}^{2,4}$ we have $\varphi\wedge \ast_S\theta  = g(\varphi, \theta)\mu_S$.
%Since $\mu= \mathbf{n}\wedge \mu_S$, we obtain
%$\mathbf{n}\wedge \varphi\wedge \ast_S\theta = g(\varphi,\theta)\mathbf{n}\wedge \mu_S = g(\varphi,\theta)\mu$, or, which is the same thing,
%$(-1)^k\varphi\wedge \mathbf{n}\wedge \ast_S\theta = g(\varphi,\theta)\mu = \varphi\wedge \ast_R\theta$.
%The last equality is true not only for the forms $\varphi$ on $T_x\mathbb{S}^{2,4}$, but also for any $k$-forms $\varphi=\varphi_S +\mathbf{n}\wedge \widetilde{\varphi}$  on $\R^7$.
%On the second component $\mathbf{n}\wedge \widetilde{\varphi}$, both sides of the equality vanish.
%Therefore, we obtain the equality $(-1)^k\mathbf{n}\wedge \ast_S\theta =\ast_R\theta$.
%Applying the inner product $\iota_\mathbf{n}$ to both sides of the last equality, we obtain the assertion of the lemma.
%\end{proof}

%%%%%%%%%%%%%%%%%%%%%%%%%%%%%%%%%%%%%%%%%%%%%%%%%%%
\begin{lemma}\label{Lemma 4}
Let $\theta$ be a $k$-form on $\R^{3,4}$ and let $\theta|_S$ be its restriction to the tangent space $T_x\mathbb{S}^{2,4}\subset \R^{3,4}$.
Then
$$
\ast_R(\theta|_S) = (-1)^k\mathbf{n}^*\wedge \ast_S(\theta|_S), \quad      \ast_S(\theta|_S) = (-1)^k\iota_\mathbf{n}(\ast_R\theta).
$$
%$$
%\ast_R\theta = (-1)^k\mathbf{n}\wedge \ast_S\theta, \quad      \ast_S(\theta|_S) = (-1)^k\iota_\mathbf{n}(\ast_R\theta).
%$$
\end{lemma}
\begin{proof}
For any $k$-form $\varphi$ on $\mathbb{S}^{2,4}$ we have $\varphi\wedge \ast_S(\theta|_S) = g(\varphi, \theta|_S)\mu_S$.
Since $\mu= \mathbf{n}^*\wedge \mu_S$, we obtain
$\mathbf{n}^*\wedge \varphi\wedge \ast_S(\theta|_S) = g(\varphi,\theta|_S)\mathbf{n}^*\wedge \mu_S = g(\varphi,\theta|_S)\mu$, or, which is the same thing,
$(-1)^k\varphi\wedge \mathbf{n}^*\wedge \ast_S(\theta|_S) = g(\varphi,\theta|_S)\mu = \varphi\wedge \ast_R(\theta|_S)$.
The last equality is true not only for the forms $\varphi$ on $T_x\mathbb{S}^{2,4}$, but also for any $k$-forms $\varphi=\varphi_S +\mathbf{n}^*\wedge \widetilde{\varphi}$  on $\R^7$.
On the second component $\mathbf{n}^*\wedge \widetilde{\varphi}$, both sides of the equality vanish.
Therefore, we obtain the equality $(-1)^k\mathbf{n}^*\wedge \ast_S(\theta|_S) =\ast_R(\theta|_S)$.
Applying the inner product $\iota_\mathbf{n}$ to both sides of the last equality, we obtain the assertion of the lemma.
\end{proof}
%%%%%%%%%%%%%%%%%%%%%%%%%%%%%%%%%%%%%%%%%%%%%%%%%%%

\begin{theorem} \label{Th-1}
The fundamental form $\omega$ of an almost complex Cayley structure $J$ on $\mathbb{S}^{2,4}$ have the following properties:
\begin{equation}
		\omega = \iota_\mathbf{n}\Omega, \qquad  d\omega=3\Omega|_S,
\end{equation}
\begin{equation}
 		\ast_S\omega = \Psi|_S, \qquad \ast_Sd\omega = -3\iota_\mathbf{n}\Psi,
\end{equation}
\begin{equation}
 		d\omega(X,Y,Z) = 3\iota_\mathbf{n}\Psi(JX,Y,Z),
\end{equation}
\begin{equation}
 		d\omega(JX,Y,Z) = d\omega(X,JY,Z) = d\omega(X,Y,JZ),
\end{equation}
\begin{equation}
 		d\omega(X,JY,JZ) = –d\omega(X,Y,Z).
\end{equation}
\begin{equation}
 		\omega\wedge d\omega = 0.
\end{equation}
\end{theorem}
\begin{proof}
The first equality follows from the definition of $\omega$ and $\Omega$: $\iota_\mathbf{n}\Omega(X,Y) = g(\mathbf{n},X\times Y) = \omega(X,Y)$, for any $X,Y$ tangent to the sphere. In the second equality, we assume that the vector field $\mathbf{n}(x) = x$ is defined on the whole of $\R^7$.
Then, since the 3-form $\Omega$ with constant coefficients in $\R^7$, it is closed, therefore $d\iota_\mathbf{n}\Omega = L_\mathbf{n}\Omega = 3\Omega$, where $L_\mathbf{n} = d\,\iota_\mathbf{n} + \iota_\mathbf{n}\,d$ is the Lie derivative along the vector field $\mathbf{n}(x) = x$ on $\R^7$.
Therefore, $d\omega = d\iota_\mathbf{n}\Omega = L_\mathbf{n}\Omega = 3\Omega$. The second pair of equalities is obtained from Lemmas \ref{Lemma 3} and \ref{Lemma 4} and from the first properties, $\ast_S\omega  = \ast_S(\iota_\mathbf{n}\Omega)|_S = \iota_\mathbf{n}\ast_R(\iota_\mathbf{n}\Omega)$.
In more detail, let $\mathbf{e}_\mathbf{n}$ be the operator of exterior multiplication by a 1-form dual to the vector field $\mathbf{n}(x) = x$.
As is known, on the $k$-forms the operators $\iota_\mathbf{n}$ and $\mathbf{e}_\mathbf{n}$ are related by the relation $\iota_\mathbf{n}\cdot\ast_R = (-1)^k\ast_R\cdot\, \mathbf{e}_\mathbf{n}$. therefore
$$
\ast_S\omega = \iota_\mathbf{n}\ast_R(\iota_\mathbf{n}\Omega) = \ast_R\mathbf{e}_\mathbf{n}(\iota_\mathbf{n}\Omega) = \ast_R\mathbf{e}_\mathbf{n}\left(\iota_\mathbf{n}(\Omega|_S+\Omega_\mathbf{n})\right) = \ast_R(\Omega_\mathbf{n}) = \Psi|_S.
$$
Further, $\ast_Sd\omega = 3\ast_S\Omega|_S = -3\iota_\mathbf{n}\ast_R\Omega = -3\iota_\mathbf{n}\Psi$.
Equality (3) follows from (1) and from Lemma 2. Properties (4-5) follow from (1) and Lemma \ref{lem-1-associator}.
The last equality follows from (1) and from $\iota_\mathbf{n}(\Omega\wedge\Omega)=0$.
\end{proof}

\begin{remark}
The fundamental 2-form $\omega$ on $\mathbb{S}^{2,4}$ is co-closed, $\delta\omega =0$.
Indeed, the 4-form $\Psi$ on $\R^7$ has constant coefficients, therefore $d\Psi = 0$. It is well known that on the $k$-forms the codifferential $\delta$ has the expression: $\delta = (-1)^k \ast^{-1}d\ast = (-1)^{kn+n+1+\eta} \ast d\ast$, where $\eta$ is the number of "minuses" of the metric tensor $g$. In our case, $\eta = 4$ and $n =6$.
Hence $\delta = -\ast d\ast$. Therefore:
$$
\delta\omega =-\ast d\ast\omega = -\ast d(\Psi|_S) = -\ast(d \Psi)|_S = 0.
$$
\end{remark}

In \cite{Hitchin}, Hitchin defined the notion of nondegeneracy (stability) for 3-forms $\Omega$ on a six-dimensional real vector space $V$ and constructed a linear operator $K_\Omega$ for the 3-form $\Omega$, whose square is proportional to the identity operator, $K_\Omega^2 = \lambda(\Omega)Id$.
If the $\lambda(\Omega) <0$, then $K_\Omega$ defines a complex structure $I_\Omega$ on $V$, and if $\lambda(\Omega) >0$, then the para-complex structure $P_\Omega$.
Let $\mu$ be a volume form on $V$, then $K_\Omega$ is defined by following formula: $\iota_{K_\Omega(X)}\mu= \iota_X\Omega\wedge\Omega$.

For the fundamental form $\omega$ of an almost complex Cayley structure on $\mathbb{S}^{2,4}$, the 3-form $d\omega$ has the form $d\omega=3\Omega|_S$, where $\Omega = \omega^{123} -\omega^{145} +\omega^{167} -\omega^{246} -\omega^{257} -\omega^{347} +\omega^{356}$.
We find the Hitchin operator $K_{d\omega}$ for the 3-form $d\omega$ on the sphere $\mathbb{S}^{2,4}$.
First consider the point $x = e_1 \in \mathbb{S}^{2,4}$.
The tangent space $T_x\mathbb{S}^{2,4}$ has an orthonormal basis of vectors $e_2,\dots, e_7$ and the volume form $\mu = \omega^{234567}$.
The 2-form $d\omega$ on $T_x\mathbb{S}^{2,4}$ has the form
$$
d\omega_x=3(-\omega^{246} -\omega^{257} -\omega^{347} +\omega^{356}).
$$
The 2-form $d\omega_x$ is invariant at the action on $T_x\mathbb{S}^{2,4}$ of the isotropy subgroup $SU(1,2)\subset G_2^*$. Therefore, it is sufficient to calculate $K_{d\omega}$ on one vector, for example, $K_{d\omega}(e_2)$.
$$
\iota_{e_2}d\omega_x\wedge d\omega_x =9(-\omega^{46} -\omega^{57})\wedge (-\omega^{246} -\omega^{257} -\omega^{347} +\omega^{356})=
$$
$$
=9(\omega^{46}\wedge\omega^{257} +\omega^{57}\wedge\omega^{246})= -18\omega^{24567} = 18\iota_{e_3}\omega^{234567}=18\iota_{e_3}\mu.
$$
Therefore, $K_{d\omega}(e_2) = 18e_3 = 18e_1\times e_2 = 18J(e_2)$.
It follows that $I_{d\omega} = J_x$.
From the $G_2^*$-invariance of the almost complex Cayley structure $J$ and the form $\Omega$ we obtain that the equality $I_{d\omega} = J$ holds at all points of the pseudo-sphere $\mathbb{S}^{2,4}$.

\begin{conclusion}
The 3-form $d\omega$ on $\mathbb{S}^{2,4}$ is everywhere nondegenerate and defines an almost complex structure on $\mathbb{S}^{2,4}$, which coincides with the almost complex structure of Cayley $J$.
\end{conclusion}

%%%%%%%%%%%%%%%%%%%%%%%%%%%%%%%%%%%%%%%%%%%%%%%%%%%%%%%

Let's calculate the Nijenhuis tensor
$$
N(X,Y) = 2([JX,JY] -[X,Y] -J[JX,Y] -J[X,JY])
$$
of the almost complex structure $J$.
To do this, we find the covariant derivative of the tensor $J$. Since $\mathbf{n}(x) = x$, then for any tangent vector $X\in T_x\mathbb{S}^{2,4}$ we have $D\mathbf{n}_x(X) = X$.
We assume that the tangent vectors $X,Y\in T_x\mathbb{S}^{2,4}$ are continued to the space $\R^7$ as constant vector fields and $\mathbf{n}(x) =x$ is also defined on $\R^7\setminus\{0\}$.
Then we have the equalities $(\nabla_XJ)Y  =\mathrm{pr}_x(D_X(JY)) = \mathrm{pr}_x(D_X(\mathbf{n}\times Y)) = \mathrm{pr}_x(X\times Y)$, where $\mathrm{pr}$ is the projection orthogonal of $\R^{3,4}$ onto the tangent space $T_x\mathbb{S}^{2,4}$: $\mathrm{pr}_x(Z) = Z -g(Z,\mathbf{n})\mathbf{n}$. We get,
\begin{equation} \label{nabla-J}
 		(\nabla_XJ)Y = X\times Y -g(\mathbf{n}, X\times Y)\mathbf{n} = X\times Y -\omega(X,Y)\mathbf{n}.
\end{equation}

\begin{remark}
From the resulting expression (\ref{nabla-J}) of $\nabla_XJ$, the skew-symmetry of $\nabla_XJ$ follows immediately and, in particular, the property $\nabla_X(J)X =0$ of the nearly pseudo-K\"{a}hler property of the variety $\mathbb{S}^{2,4}$ and the equality $\nabla_{JX}J =-J\nabla_XJ$.
\end{remark}

\begin{theorem}\label{Th-2}
The Nijenhuis tensor $N(X,Y)$ of the almost complex Cayley structure $J$ on $\mathbb{S}^{2,4}$ has the form
$$
N(X,Y) = -8\,\mathbf{n}\times (X\times Y)= -8\,J(X\times Y).
$$
\end{theorem}
\begin{proof}
The book \cite{KN} gives the formula
$$
4g((\nabla_ZJ)X,Y) = 6\,\mathbf{d}\Phi(Z,JX,JY) -6\,\mathbf{d}\Phi(Z,X,Y)+ g(N(X,Y),JZ),
$$
where $\Phi(X,Y)= g(X,JY)$ is the fundamental form.
In our case $\omega(X,Y)= g(JX,Y) = -\Phi(X,Y)$.

In this work, we assume that exterior product and exterior differential are defined without normalizing constant. In particular, then $dx\wedge dy = dx\otimes dy - dy\otimes dx$ and $d\eta (X,Y) = X\eta(Y) -Y\eta(X) -\eta([X,Y])$. In \cite{KN} another definition of external multiplication is taken. The external differential $\mathbf{d}$ in \cite{KN} is taken with a coefficient 1/3: $\mathbf{d}\omega = d\omega/3$.
Taking this into account, we have in our case
$$
4g((\nabla_ZJ)X,Y) = -2d\omega(Z,JX,JY) +2d\omega(Z,X,Y)+ g(N(X,Y),JZ).
$$
Therefore $g(N(X,Y),JZ) = 4g((\nabla_ZJ)X,Y) + 2d\omega(Z,JX,JY) -2d\omega(Z,X,Y)$.
Then, taking into account that $d\omega(X,Y,Z) = 3g(X,Y\times Z)$, we have:
$$
g(N(X,Y),JZ) = 4g((\nabla_ZJ)X,Y) +2d\omega(Z,JX,JY) -2d\omega(Z,X,Y) =
$$
$$
=4g(Z\times X-2\omega(Z,X)\mathbf{n},Y)-2d\omega (X,Y,Z) -2d\omega (Z,X,Y) =
$$
$$
= 4g(Z\times X,Y) -12g(X,Y\times Z) = 4g(Z, X\times Y) -12g(X\times Y,Z) =
$$
$$
= -8g(\mathbf{n}\times (X\times Y), \mathbf{n}\times Z) = -8g(\mathbf{n}\times (X\times Y), JZ).
$$
\end{proof}

\begin{theorem}\label{Th-3}
The fundamental 2-form $\omega$ of the almost complex Cayley structure on $\mathbb{S}^{2,4}$ is proper for the Laplace operator: $\Delta\omega = 12\omega$.
\end{theorem}
\begin{proof}
The Laplacian acting on differential forms is defined by the formula $\Delta = d\delta + \delta d$, where $\delta$ is the codifferential.
It is well known that on the $k$-forms the codifferential $\delta$ has the expression: $\delta = (-1)^k \ast^{-1}d \ast = (-1)^{kn+n+1+\eta} \ast d \ast$, where $\eta$ is the number of "minuses" of the metric tensor $g$.
In our case, $\eta =4$ and $n =6$.
Therefore, $\delta = -\ast d \ast$.
For the form $\omega$ it is known that $\delta\omega =0$.
Therefore, it suffices to compute $\delta d\omega$.
We denote the exterior differential in the space $\R^7$ with the index: $d_R$.
We use the property $\ast_S(\theta|_S) = (-1)^k\iota_\mathbf{n}(\ast_R \theta)$, established in Lemma \ref{Lemma 4}, and the results of Theorem \ref{Th-1}.
$$
\delta d\omega = -\ast_Sd\ast_S(d\omega) = -3\ast_Sd\ast_S(\Omega|_S) =
$$
$$
= -3(-1)^3\ast_Sd(\iota_\mathbf{n}\ast_R\Omega) = 3\ast_S(d_R\iota_\mathbf{n}(\Psi))|_S = 3\ast_S(L_\mathbf{n}\Psi)|_S =
$$
$$
=12\ast_S(\Psi|_S) = 12\,\iota_\mathbf{n}\ast_R\Psi  =12\,\iota_\mathbf{n}\Omega = 12\,\omega.
$$
In these calculations, we used the Lie derivative $L_\mathbf{n} = d\cdot \iota_\mathbf{n} + \iota_\mathbf{n}\cdot d$ along the vector field $\mathbf{n}(x) = x$ on $\R^7\setminus\{0\}$.
The 4-form $\Psi$ with constant coefficients in $\R^7$ is closed, therefore $d\iota_\mathbf{n}\Psi =L_\mathbf{n}\Psi$. In addition, the field $\mathbf{n}(x) = x$ defines the dilatation $\R^7\setminus\{0\}$: $x\mapsto tx$. Therefore, for the 4-form $\Psi$ with constant coefficients, the equality $L_\mathbf{n}\Psi  = 4\Psi$ holds.
\end{proof}

\subsection{The complex structure on $\mathbb{S}^{2,4}$} \label{CS-on-S24}
We have shown that the almost complex Cayley structure on the pseudo sphere $\mathbb{S}^{2,4}$ is non-integrable.
In this subsection we show that there are integrable complex structures on the pseudosphere $\mathbb{S}^{2,4}$.
Consider the stereographic projection $\mathbb{S}^{2,4}$ from the origin of coordinates $\R^7$ to the part of the cylinder $\mathbb{S}^2\times D^4$, where $\mathbb{S}^2$ is the standard unit sphere in the space $\R^3$ with the coordinates $(x_1, x_2, x_3)$ and $D^4$ is the unit ball in the space $\R^4$ with the coordinates $(x_4, x_5, x_6, x_7)$, $x_4^2+ x_5^2+ x_6^2+ x_7^2 <1$. The line passing through the origin and the point $x = (x_1, x_2,\dots, x_7)$ of the pseudosphere $x_1^2+ x_2^2+ x_3^2 -x_4^2 -x_5^2 -x_6^2 -x_7^2 =1$, intersect the cylinder $\mathbb{S}^2\times D^4$  at the point $y$ with the coordinates (Fig. 1)
$$
y_i=\frac{x_i}{\sqrt{x_1^2+ x_2^2+ x_3^2}},\quad i=1,2,\dots, 7.
$$
It is clear that the first variables $(y_1, y_2, y_3)$ describe the two-dimensional sphere $\mathbb{S}^{2}$ in $\R^3$.
The remaining coordinates vary in the unit ball $D^4$ in the space $\R^4$. Indeed, let $r^2 =x_1^2+ x_2^2+ x_3^2$.
On the pseudosphere $r^2\geq 1$.
Then it follows from the pseudo sphere equation that
$$
\frac{x_4^2 +x_5^2 +x_6^2 +x_7^2}{x_1^2+ x_2^2+ x_3^2} =1-\frac{1}{x_1^2+ x_2^2+ x_3^2}, \quad \Rightarrow \quad 0\leq y_4^2 +y_5^2 +y_6^2 +y_7^2 = 1-\frac{1}{r^2}<1.
$$

Through each point $(y_4, y_5, y_6, y_7)\in D^4$, at a fixed point $(y_1, y_2, y_3)\in \mathbb{S}^2$, there passes a single straight line intersecting the pseudo sphere at the point $x_i=r\, y_i$, $i=1,2,\dots, 7$ at $r=1/\sqrt{1-y_4^2 -y_5^2 -y_6^2 -y_7^2}$ and vice versa, each straight line passing through the origin and the point $x\in \mathbb{S}^{2,4}$ intersects the ball $D^4$ for certain $(y_1, y_2, y_3)\in \mathbb{S}^{2}$.

\begin{figure}\centering
\includegraphics[scale=1.0]{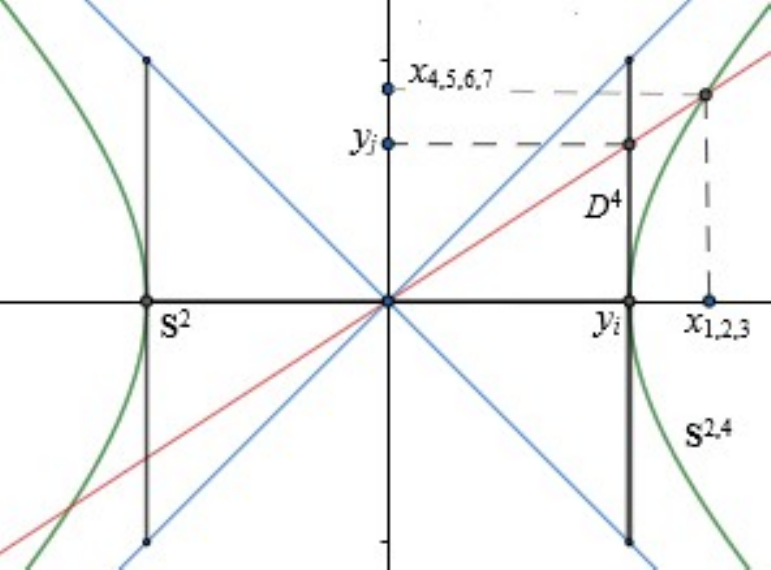}\\ %{Fig-1.bmp}
\caption{Stereographic projection} \label{Fg-1}
\end{figure}

Since the direct product $\mathbb{S}^2\times D^4$ admits complex structures, the diffeomorphism of the stereographic projection $f: \mathbb{S}^{2,4} \rightarrow \mathbb{S}^2\times D^4$ allows us to transfer the complex structures from $\mathbb{S}^2\times D^4$ to the sphere $\mathbb{S}^{2,4}$.
We note that $\mathbb{S}^2\times D^4$ is a pseudo-Riemannian signature (2,4).

\section{Almost para-complex structure on $\mathbb{S}^{3,3}$} \label{APCS-on-S33}
The results of this part are similar to those obtained for the sphere $\mathbb{S}^{2,4}$, but there are some differences.
Consider the sphere $\mathbb{S}^{3,3} =\mathbb{S}^{3,3}(i)$ of a purely imaginary unit radius in the space $\R^{3,4} = \Im(\mathbf{Ca'})$ of imaginary Cayley split-octaves.
The tangent planes to the sphere $T_x\mathbb{S}^{3,3}$, $\forall x\in \mathbb{S}^{3,3}$, are pseudo-Euclidean signatures (3.3).
Scalar squares of the normal vectors $\mathbf{n}(x) = x$ to the sphere $\mathbb{S}^{3,3}$ are negative, $g(\mathbf{n},\mathbf{n}) = -1$.
Vector cross multiplication of the tangent vectors $Y\in T_x\mathbb{S}^{3,3}$ by the normal vector $\mathbf{n}(x)$ takes the tangent space into itself: if $Y \in T_x\mathbb{S}^{3,3}$, that is, $g(Y,\mathbf{n}) = 0$, then
$$
g(\mathbf{n}\times Y,\mathbf{n}) = 0, \qquad \mathbf{n}\times Y \in T_x\mathbb{S}^{3,3}.
$$
It follows from the property of the vector cross product that
$$
\mathbf{n}\times (\mathbf{n}\times Y) = -g(\mathbf{n},\mathbf{n})Y +g(\mathbf{n},Y)\mathbf{n} = Y.
$$
This means that $(P_x)^2 = Id$.
The operator $P_x(X) = \mathbf{n}\times X$ is skew-symmetric, $g(\mathbf{n}\times X,Y) = -g(X,\mathbf{n}\times Y)$, but not orthogonal: $g(P_xX,P_xY) = g(\mathbf{n}\times X,\mathbf{n}\times Y) = -g(X,Y)$, $\forall  X,Y \in T_x\mathbb{S}^{3,3}$.
The operator $P_x$ map space-like vectors into time-like vectors.
The eigensubspaces of the operator $P_x$ are three-dimensional isotropic subspaces tangent to $\mathbb{S}^{3,3}$.
Therefore, the operator $P_x(X) = \mathbf{n}(x)\times X$ defines a para-complex structure on $T_x\mathbb{S}^{3,3}$.
The $P_x$ is defined at each point $x\in \mathbb{S}^{3,3}$.
Hence we get that the unit pseudo-Riemann sphere $\mathbb{S}^{3,3}$ has a natural almost para-complex structure $P$, which we call an \emph{almost para-complex Cayley} structure:
$$
P_x: T_x \mathbb{S}^{3,3} \rightarrow T_x \mathbb{S}^{2,4},\quad P_x(Y) = \mathbf{n}(x)\times Y, \quad \forall x\in \mathbb{S}^{3,3}, \forall  Y \in T_x\mathbb{S}^{3,3}.
$$

Let $\omega(X,Y) = g(PX,Y)$ be the fundamental 2-form corresponding to $P$ and let $\Omega(X,Y,Z) = g(X\times Y,Z)$.
Then it is easy to see that
$$
\omega(X,Y) = g(\mathbf{n}\times X,Y) = \Omega(\mathbf{n},X,Y) = g(\mathbf{n},X\times Y).
$$
Recall that an exterior 4-form $\Psi = \ast \Omega$ is defined on the space $\R^{3,4} = \Im(\mathbf{Ca'})$.
Let $\ast_S$ and $\ast_R$ be the Hodge operators on the sphere $\mathbb{S}^{3,3}$ and on $\R^{3,4}$, respectively.
The symbol $\theta|_S$ we denote the restriction of the differential form $\theta$ on $\R^{3,4}$ to the submanifold $\mathbb{S}^{3,3}$, and by the symbol $\theta_n$ is the normal component of $\theta$ along $\mathbb{S}^{3,3}$, then $\theta=\theta|_S + \mathbf{n}^*\wedge \widetilde{\theta}$.

For the almost para-complex Cayley structure $P$, all results of Section \ref{ACS-on-S24} for almost complex Cayley structures $J$, are completely similarly proved:
\begin{enumerate}
  \item The 3-form $\Omega$ under its restriction to a sphere has the property:
$$
\Omega(PX,Y,Z) = \Omega(X,PY,Z) = \Omega(X,Y, PZ).
$$
  \item For any $X,Y,Z \in T_x\mathbb{S}^{3,3}$ we have the equality:
$$
\iota_\mathbf{n}\Psi(X,Y,Z) = -\Omega(PX,Y,Z).
$$
  \item Let $\theta$ be a $k$-form on $\R^{3,4}$ and let $\theta|_S$ be its restriction to the tangent space $T_x\mathbb{S}^{3,3} \subset \R^{3,4}$. Then
$$
\ast_R(\theta|_S) = (-1)^k\mathbf{n}^*\wedge \ast_S(\theta|_S), \quad      \ast_S(\theta|_S) = (-1)^k\iota_\mathbf{n}(\ast_R\theta).
$$
  \item The fundamental form $\omega$ of an almost para-complex Cayley structure $P$ on $\mathbb{S}^{3,3}$ and its exterior differential $d\omega$  have the following properties:
$$
	\omega = \iota_\mathbf{n}\Omega, \quad  d\omega=3\Omega|_S, \quad
	\ast_S\omega = \Psi|_S, \qquad \ast_Sd\omega = -3\iota_\mathbf{n}\Psi, \quad \omega\wedge d\omega = 0.
$$
$$
 	d\omega(X,Y,Z) = 3\iota_\mathbf{n}\Psi(JX,Y,Z), \quad d\omega(X,Y,Z) = 3g(X,Y\times Z)
$$
$$ 		d\omega(PX,Y,Z) = d\omega(X,PY,Z) = d\omega(X,Y,PZ), \quad
 		d\omega(X,PY,PZ) = -d\omega(X,Y,Z).
$$
\end{enumerate}

For the fundamental form $\omega$ of an almost para-complex Cayley structure on $\mathbb{S}^{3,3}$, the 3-form $d\omega$ has the form $d\omega=3\Omega|_S$, where $\Omega = \omega^{123} -\omega^{145} +\omega^{167} -\omega^{246} -\omega^{257} -\omega^{347} +\omega^{356}$.
We find the Hitchin operator \cite{Hitchin} $K_{d\omega}$ for the 3-form $d\omega$ on the sphere $\mathbb{S}^{3,3}$.
First consider the point $x = e_4 \in \mathbb{S}^{3,3}$.
The tangent space $T_x\mathbb{S}^{3,3}$ has an orthonormal basis and the volume form $\mu = \omega^{123567}$.
The 2-form $d\omega$ on $T_x\mathbb{S}^{3,3}$ has the form
$$
d\omega_x=3(\omega^{123} +\omega^{167} -\omega^{257} +\omega^{356}).
$$
The 2-form $d\omega_x$ is invariant at the action on $T_x\mathbb{S}^{3,3}$ of the isotropy subgroup $SL(3,\R)\subset G_2^*$. Therefore, it is sufficient to calculate $K_{d\omega}$ on one vector, for example, $K_{d\omega}(e_1)$.
$$
\iota_{e_1}d\omega_x\wedge d\omega_x =9(\omega^{23} +\omega^{67})\wedge (\omega^{123} +\omega^{167} -\omega^{257} +\omega^{356})=
$$
$$
=9(\omega^{12367} +\omega^{12367})= 18\omega^{12367} = -18\iota_{e_5}\omega^{123567}=-18\iota_{e_5}\mu.
$$
Therefore, $K_{d\omega}(e_1) = -18e_5 = 18e_4\times e_1 = 18P(e_1)$.
It follows that $I_{d\omega} = P_x$.

From the invariance of the almost para-complex Cayley structure $P$ and the 3-form $\Omega$ with respect to the group $G_2^*$, acting transitively on $\mathbb{S}^{3,3}$, we obtain that the equality $I_{d\omega} = P_x$ holds at all points of the pseudo-sphere $\mathbb{S}^{3,3}$.

\begin{conclusion}
The 3-form $d\omega$ on $\mathbb{S}^{3,3}$ is everywhere nondegenerate and defines an almost para-complex structure on $\mathbb{S}^{3,3}$, which coincides with the almost complex structure of Cayley $P$.
\end{conclusion}

Let's calculate the Nijenhuis tensor
$$
N_P(X,Y) = 2([X,Y]+[PX,PY] -P[PX,Y] -P[X,PY])
$$
of the almost para-complex structure $P$ \cite{Aleks-09}.
To do this, we find the covariant derivative of the tensor $P$. Since $\mathbf{n}(x) = x$, then for any tangent vector $X\in T_x\mathbb{S}^{3,3}$ we have $D\mathbf{n}_x(X) = X$.
We assume that the tangent vectors $X,Y\in T_x\mathbb{S}^{3,3}$ are continued to the space $\R^7$ as constant vector fields and $\mathbf{n}(x) =x$ is also defined on $\R^7\setminus\{0\}$.

Then we have the equalities $(\nabla_XP)Y  =\mathrm{pr}_x(D_X(PY)) = \mathrm{pr}_x(D_X(\mathbf{n}\times Y)) = \mathrm{pr}_x(X\times Y)$, where $\mathrm{pr}_x(Z) = Z+g(Z,\mathbf{n})\mathbf{n}$ is the projection orthogonal of $\R^{3,4}$ on the tangent space $T_x\mathbb{S}^{3,3}$, here we take into account that $g(\mathbf{n},\mathbf{n}) =-1$. We get,
\begin{equation} \label{nabla-P}
 		(\nabla_XP)Y = X\times Y +g(\mathbf{n}, X\times Y)\mathbf{n} = X\times Y +\omega(X,Y)\mathbf{n}.
\end{equation}

\begin{remark}
The skew-symmetry of the operator $\nabla_XP$ follows immediately from the resulting expression (\ref{nabla-P}) and the property $\nabla_X(P)X =0$ of the nearly pseudo-para-K\"{a}hler property of the manifold $\mathbb{S}^{3,3}$ and the equality $\nabla_{PX}P =-P\nabla_XP$.
\end{remark}

\begin{theorem}\label{Th-4-NP}
The Nijenhuis tensor $N_P(X,Y)$ of the almost para-complex Cayley structure $P$ on $\mathbb{S}^{3,3}$ has the form
$$
N_P(X,Y) = -8\,\mathbf{n}\times (X\times Y)=-8\,P(X\times Y).
$$
\end{theorem}
\begin{proof}
Direct calculation using the formula $(\nabla_XP)Y = X\times Y +\omega(X,Y)\mathbf{n}$ and the equality $(\mathbf{n}\times X)\times Y = -\mathbf{n}\times (X\times Y) -g(X,Y)\mathbf{n}$ for $X,Y \perp \mathbf{n}$:
$$
N_P(X,Y)/2 = [X,Y] + [PX,PY] -P[PX,Y] -P[X,PY]=
$$
$$
=\nabla_X(Y) -\nabla_Y(X) +\nabla_{PX}(PY) -\nabla_{PY}(PX) -P\nabla_{PX}(Y) + P\nabla_Y(PX) -P\nabla_X(PY) + P\nabla_{PY}(X) =
$$
$$
=  \nabla_{PX}(P)Y -\nabla_{PY}(P)X + P\nabla_Y(P)X -P\nabla_X(P)Y =
$$
$$
=(PX)\times Y +\omega(PX,Y)\mathbf{n} -((PY)\times X +\omega(PY,X)\mathbf{n}) +P(Y\times X +\omega(Y,X)\mathbf{n}) -P(X\times Y +\omega(X,Y)\mathbf{n}) =
$$
$$
= (\mathbf{n}\times X)\times Y -(\mathbf{n}\times Y)\times X + \mathbf{n}\times (Y\times X) -\mathbf{n}\times (X\times Y) = -4 \mathbf{n}\times (X\times Y)
$$
\end{proof}

\begin{theorem}\label{Th-5-Delta-P}
The fundamental 2-form $\omega$ of the almost para-complex Cayley structure on $\mathbb{S}^{3,3}$ is proper for the Laplace operator: $\Delta\omega = -12\omega$.
\end{theorem}
\begin{proof}
Completely repeats the proof of Theorem \ref{Th-3} with one difference: $\delta = +\ast d \ast$.
Indeed, on the $k$-forms the codifferential $\delta$ has the expression: $\delta = (-1)^k \ast^{-1}d \ast = (-1)^{kn+n+1+\eta} \ast d \ast$, where $\eta$ is the number of "minuses" of the metric tensor $g$.
In our case, $\eta =3$ and $n =6$.
Therefore, $\delta = +\ast d \ast$.
\end{proof}

\subsection{The para-complex structure on $\mathbb{S}^{3,3}$} \label{PCS-on-S33}
We have shown that the almost para-complex Cayley structure on the pseudo sphere $\mathbb{S}^{3,3}$ of a purely imaginary unit radius is non-integrable.
In this subsection we show that there are integrable para-complex structures on the pseudosphere $\mathbb{S}^{3,3}$.
This follows from the fact that $\mathbb{S}^{3,3}$ is diffeomorphic to the direct product of three-dimensional manifolds.

Consider the stereographic projection $\mathbb{S}^{3,3}$ from the origin of coordinates $\R^7$ to the part of the cylinder $D^3\times \mathbb{S}^3$, where $\mathbb{S}^3$ is the standard unit sphere in the space $\R^4$ with the coordinates $(x_4, x_5, x_6, x_7)$ and $D^3$ is the unit ball in the space $\R^3$ with the coordinates $(x_1, x_2, x_3)$, $x_1^2+ x_2^2+ x_3^2 <1$.
The line passing through the origin and the point $x = (x_1, x_2,\dots, x_7)$ of the pseudosphere $-x_1^2- x_2^2- x_3^2 +x_4^2 +x_5^2 +x_6^2 +x_7^2 =1$, intersect the cylinder $D^3\times \mathbb{S}^3$  at the point $y$ with the coordinates
$$
y_i=\frac{x_i}{\sqrt{x_4^2+ x_5^2+ x_6^2+ x_7^2}},\quad i=1,2,\dots, 7.
$$
It is clear that the last variables $(y_4, y_5, y_6, y_7)$ describe the three-dimensional sphere $\mathbb{S}^{3}$.
The remaining coordinates vary in the unit ball $D^3$ in the space $\R^3$. Indeed, let $\rho^2 =x_4^2+ x_5^2+ x_6^2+ x_7^2$.
On the pseudosphere $\rho^2\geq 1$.
Then it follows from the pseudo sphere equation that
$$
\frac{x_1^2 +x_2^2 +x_3^2}{x_4^2+ x_5^2+ x_6^2+ x_7^2} =1-\frac{1}{x_4^2+ x_5^2+ x_6^2+ x_7^2}, \quad \Rightarrow \quad 0\leq y_1^2 +y_2^2 +y_3^2 = 1-\frac{1}{\rho^2}<1.
$$
Through each point $(y_1, y_2, y_3)\in D^3$, at a fixed point $(y_4, y_5, y_6, y_7)\in \mathbb{S}^3$, there passes a single straight line intersecting the pseudo sphere at the point $x_i=\rho\, y_i$, $i=1,2,\dots, 7$ with $\rho=1/\sqrt{1-y_1^2 -y_2^2 -y_3^2}$ and vice versa, each straight line passing through the origin and the point $x\in \mathbb{S}^{3,3}$ intersects the ball $D^3$ for certain $(y_1, y_2, y_3)$.

Since the direct product $D^3\times \mathbb{S}^3$ admits para-complex structures, the diffeomorphism of the stereographic projection $f: \mathbb{S}^{3,3} \rightarrow D^3\times \mathbb{S}^3$ allows us to transfer the para-complex structures from the direct product of three-dimensional manifolds $D^3\times \mathbb{S}^3$ to the sphere $\mathbb{S}^{3,3}$.
We note that by construction, $D^3\times \mathbb{S}^3$ is a pseudo-Riemannian signature (3.3).

\section{Almost complex and almost para-complex structures on the Cayley algebra} \label{ACS-Cayley}
The space $\R^{7}$ of purely imaginary numbers $X = x_1e_1 + x_2e_2 +\dots + x_7e_7$ is divided by the cone $N_0(X)=x_1^2+ x_2^2+ x_3^2 -x_4^2 -x_5^2 -x_6^2 -x_7^2 =0$  into two open sets of $\R^{7+}$ and $\R^{7-}$  -- elements for which $N_0(X) >0$ and $N_0(X) <0$, respectively.
We consider the direct products of these sets with the real line $\R_0$ of the Cayley algebra:
$$
\R^{8+} = \R_0\times \R^{7+} \quad \text{ and }\quad  \R^{8-} = \R_0\times \R^{7-}.
$$
In this section we define a non-integrable almost complex structure on an open set $\R^{8+}$ and a non-integrable para-complex structure on $\R^{8-}$.

\subsection{Almost complex Cayley structure}
Consider the space $\R^{8+} = \R_0\times \R^{7+}$ and define a non-integrable almost complex structure $J$ on it.
On the space $\R_0$, there is defined a unit vector field $\mathbf{n}_1 = e_0 =1$.
We define the unit vector field $\mathbf{n}_2(u)$ on $\R^{8+}$ by the formula  $\mathbf{n}_2(u)=\frac{U}{||U||}$, where $u =u_0+U\in \R^{8+} = \R_0\times \R^{7+}$ and $||U||=\sqrt{N_0(U)}>0$.
Note that $\langle \mathbf{n}_2, \mathbf{n}_2 \rangle = 1$, and $\mathbf{n}_2\mathbf{n}_2 = -1$.
We define the operator $J: \R^8 \rightarrow \R^8$ at the point $u =u_0+U\in \R^{8+} = \R_0\times \R^{7+}$ by formula
$$
J(y)= \mathbf{n}_2\cdot y = \frac{U}{||U||}\cdot y.
$$
Since $\mathbf{n}_2\mathbf{n}_2 = -1$, then $J^2 = -Id$ and we obtain an almost complex structure at the point $u =u_0+U\in \R^{8+}$.

The fundamental form $\omega(X,Y) = g(JX,Y)$ corresponding to $J$ can be easily calculated taking into account the property $\mathbf{n}_2\times Y = \mathbf{n}_2\cdot Y+ \langle \mathbf{n}_2, Y \rangle$:
$$
\omega (y,z) =\langle J(y), z\rangle = \langle \mathbf{n}_2\cdot y,z\rangle = \langle \mathbf{n}_2\cdot (y_0+Y),z_0+Z \rangle = y_0\langle \mathbf{n}_2,z_0+Z \rangle + \langle \mathbf{n}_2\cdot Y,z_0+Z \rangle=
 $$
$$
= y_0\langle \mathbf{n}_2,z \rangle + \langle \mathbf{n}_2\times Y,Z \rangle -z_0\langle \mathbf{n}_2,y \rangle = \mathbf{n}_1^*\wedge \mathbf{n}_2^*(y,z) + \langle \mathbf{n}_2\times Y,Z \rangle\, .
 $$

It is also easy to calculate the derivative $D_x(J)$ of the tensor field $J$ in the direction of the vector $x = x_0 +X \in \R^{8+}$.
Since the space $\R^{8+}$ is an open subset of $\R^8$, we can consider vectors $x, y$ with parallel vector fields on $\R^{8+}$, and then $D_x(J)(y) = D_x(\mathbf{n}_2)y$.

The vector field $\mathbf{n}_2(u)$ does not depend on the variable $u_0$, so its derivative in the direction of the vector $x = e_0$ is zero.
Therefore, it is sufficient to calculate $D_x(\mathbf{n}_2)$ only in the direction of the vector $x = X\in \R^{7+}$.
This derivative is found by simple differentiation of the unit vector field $\mathbf{n}_2(u)$.
If $u=u_0+U \in \R^{8+}$, then $\mathbf{n}_2(u)=\frac{U}{||U||}$.
Therefore, from the formula $D_x(J)(y) = D_x(\mathbf{n}_2)y$ we obtain the following expression for the derivative of the almost complex Cayley structure $J$ on $\R^{8+}$ at the point $u=u_0+U$ in the direction of the vector $x =x_0 +X$:
$$
D_x(J)(y) = (D_X\mathbf{n}_2)y=\frac{1}{||U||}\left(X-\langle X,\mathbf{n}_2(u) \rangle \mathbf{n}_2(u) \right)\cdot y
$$

Using this expression, we can easily calculate (like in Theorem \ref{Th-4-NP}) the Nijenhuis tensor $N(x,y)$ for the vectors $x, y$ orthogonal to $\mathbf{n}_2(u)$.
We compute the exterior differential of the 2-form $\omega$ at the point $u=u_0+U$ using the obtained formula for $D_x(J)$:
$$
d\omega(x,y,z) =x\omega(y,z)+y\omega(z,x)+z\omega(x,y)=
$$
$$
=\langle D_x(J)y,z\rangle +\langle D_y(J)z,x\rangle+\langle D_z(J)x,y\rangle=
$$
$$
=\frac{1}{||U||}\left(3\langle X\times Y,Z\rangle
-\langle X,\mathbf{n}_2(u)\rangle\omega(Y,Z)
-\langle Y,\mathbf{n}_2(u)\rangle\omega(Z,X)
-\langle Z,\mathbf{n}_2(u)\rangle\omega(X,Y) \right) =
$$
$$
=\frac{1}{||U||}\left(3\Omega(X,Y,Z)
-(\mathbf{n}_2^*\wedge\omega)(X,Y,Z) \right).
$$

\subsection{Almost para-complex Cayley structure}
Consider the space $\R^{8-} = \R_0\times \R^{7-}$ and define a non-integrable almost para-complex structure $P$ on it.
On the line $\R_0$ we define a unit vector field $\mathbf{n}_1 = e_0 =1$.
We define the unit vector field $\mathbf{n}_2(u)$ on $\R^{8-}$ by the formula  $\mathbf{n}_2(u)=\frac{U}{||U||}$, where $u =u_0+U\in \R^{8-} = \R_0\times \R^{7-}$ and $||U||=\sqrt{-N_0(U)}>0$.
Note that $\langle \mathbf{n}_2, \mathbf{n}_2 \rangle = -1$, and $\mathbf{n}_2\mathbf{n}_2 = +1$.
We define the operator $P: \R^8 \rightarrow \R^8$ at the point $u =u_0+U\in \R^{8-} = \R_0\times \R^{7-}$ by the formula
$$
J(y)= \mathbf{n}_2\cdot y = \frac{U}{||U||}\cdot y.
$$
It is easy to see that $P^2 = Id$.
Therefore, we obtain a non-integrable almost para-complex structure on $\R^{8-}$.
In a completely analogous way, we compute $D_x(P)$, the fundamental form $\omega(X,Y) = g(PX,Y)$ associated with $P$ and its exterior differential at the point $u = u_0 + U$.
They have exactly the same form as in the case of an almost complex structure on $\R^{8+}$:
$$
\omega (y,z) =\mathbf{n}_1^*\wedge \mathbf{n}_2^*(y,z) + \langle \mathbf{n}_2\times Y,Z \rangle ,
$$
$$
d\omega(x,y,z) =
\frac{1}{||U||}\left(3\Omega(X,Y,Z)
-(\mathbf{n}_2^*\wedge\omega)(X,Y,Z) \right).
$$

\end{document}